\documentclass{amsart}
\usepackage[utf8]{inputenc}
\usepackage[T1]{fontenc}
\usepackage[english]{babel}
\usepackage{amsrefs,amssymb,amsmath,amsthm,amsfonts,graphicx}

\usepackage[colorlinks=true]{hyperref}
\hypersetup{
    citecolor=blue,
    filecolor=red,
    linkcolor=blue,
    urlcolor=blue
}

\usepackage{lineno}

\theoremstyle{plain}
\newtheorem{thm}{Theorem}[section]
\newtheorem{cor}[thm]{Corollary}
\newtheorem{lem}[thm]{Lemma}
\newtheorem{prop}[thm]{Proposition}
\theoremstyle{definition}
\newtheorem{df}[thm]{Definition}
\theoremstyle{remark}
\newtheorem{rmk}[thm]{Remark}
\newtheorem{ejp}[thm]{Example}

\DeclareMathOperator{\diam}{diam}

\DeclareMathOperator{\dist}{dist}

\DeclareMathOperator{\wlim}{L^+}
\DeclareMathOperator{\alim}{L^{--}}

\newcommand{\hyper}[1]{{\mathcal K}(#1)}
\newcommand{\continua}[1]{{\mathcal C}(#1)}

\newcommand{\radio}{\rho}

\newcommand{\tra}{T}

\newcommand{\R}{\mathbb R}

\newcommand{\Z}{\mathbb Z}

\newcommand{\reparam}{\mathcal{R}}

\renewcommand{\epsilon}{\varepsilon}

\newcommand{\tauu}{\hat\tau}
\newcommand{\tu}{\hat t}

\title[Fields of Cross Sections and Expansive Flows]{Discrete and Continuous Topological Dynamics: Fields of Cross Sections and 
Expansive Flows}

\author[A. Artigue]{Alfonso Artigue}
\email{artigue@unorte.edu.uy}
\address{DMEL, Universidad de la Rep\'ublica, Uruguay}
\begin{document}


\begin{abstract}
In this article we consider the general problem of translating 
definitions and results from 
the category of discrete-time dynamical systems to the category of flows.
We consider the dynamics of homeomorphisms and flows on compact metric spaces, 
in particular Peano continua.
As a translating tool, we construct continuous, symmetric and monotonous 
fields of local cross sections for an arbitrary flow without singular points. 
Next, we use this structure in the study of expansive flows on Peano continua.
We show that expansive flows admit no stable point and 
that every point contains a non-trivial continuum in its stable set. 
As a corollary we obtain that no 
Peano continuum with an open set homeomorphic with the plane
admits an expansive flow.
In particular compact surface admits 
no expansive flow without singular points.																																																																																																																																																																																																																																																																																																																																																																																																																																																																																																																																																																																																																																																																																																																																																																																																																																																																																																																																																																																																																																																																																																																																																																																																																																																																																																																																																																																																																																																																																																																																																																																																																																																																																																																																																																																																																																																																																																																																																																																																																																																																																																																																																																																																																																																																																																																																																																																																																																																																																																																																																																																																																																																																																																																																																																																																																																																																																																																																																																																																																																																																																																																																																																																																																																																																																																																																																																																																																																																																																																																																																																																																																																																																																																																																																																																																																																																																																																																																																																																																																																																																																																																																																																																																																						
\end{abstract}
\maketitle
\section{Introduction}

In the study of dynamical systems
\emph{time} is usually modeled as discrete or continuous. 
A discrete-time dynamical system can be understood as an action of the integers, 
induced by a homeomorphism or a diffeomorphism. 
The continuous time can be represented by the real 
numbers, and the dynamics is generated by a vector field or a flow.
Both categories are strongly related and several definitions and 
results can be \emph{translated}. 
For example, expansivity is defined in both categories. 
Recall that a homeomorphism $f\colon X\to X$ of a compact metric space 
$(X,\dist)$ is an \emph{expansive homeomorphism} if there is 
$r>0$ such that $\dist(f^n(x),f^n(y))\leq r$ for all $n\in \Z$ implies $x=y$.
For flows, the definition of \cite{BW} is a good translation 
in terms of the results that it allows to recover. 
It saids that a flow $\phi\colon\R\times X\to X$ is 
an \emph{expansive flow} if for all 
$\epsilon>0$ there is $\delta>0$ such that if 
$\dist(\phi_{h(t)}(y),\phi_t(x))<\delta$ for all $t\in\R$ with 
$h\colon\R\to\R$ an increasing homeomorphism such that 
$h(0)=0$ then there is $t\in(-\epsilon,\epsilon)$ such that $y=\phi_t(x)$.
The complexity of the definition makes the translations very difficult, 
but several techniques were developed. 
See for example \cites{KS, Th87, MSS, Pat, Oka}. 



The purpose of the present article is to develop the technique 
of local cross sections
for translating some definitions and results from discrete 
to continuous topological dynamical systems on compact metric spaces.
In order to describe our results let us consider a 
non-singular smooth flow $\phi_t\colon M\to M$, $t\in \R$, 
on a compact manifold $M$ with a transverse foliation. 
Given a point $x\in M$ 
we can take a compact disc $H_r(x)$ of radius $r>0$ 
contained in the leaf of $x$. 
If $r$ is small we have that: 
\begin{enumerate}
  \item for each $r>0$ and $x\in M$, $H_r(x)$ 
  is a local cross section containing $x$, 
  \item the map $(r,x)\mapsto H_r(x)$ is continuous,\footnote{In a smooth category, we can have continuity in the space of $C^1$ embeddings. In the context of metric spaces we will consider the Hausdorff metric between compact subsets of $X$.}
  \item (monotonous) for $t\neq 0$ small $H_r(x)$ is disjoint from $H_r(\phi_t(x))$,
  \item (symmetric) if $x,y$ are close then $y\in H(x)$ if and only if $x\in H(y)$.
\end{enumerate}
Such a map $H$, depending on $r$ and $x$, 
is an example of what we call \emph{field of cross sections}
and it is illustrated in Figure \ref{figSecs}. 

\begin{figure}[h]
  \center{\includegraphics{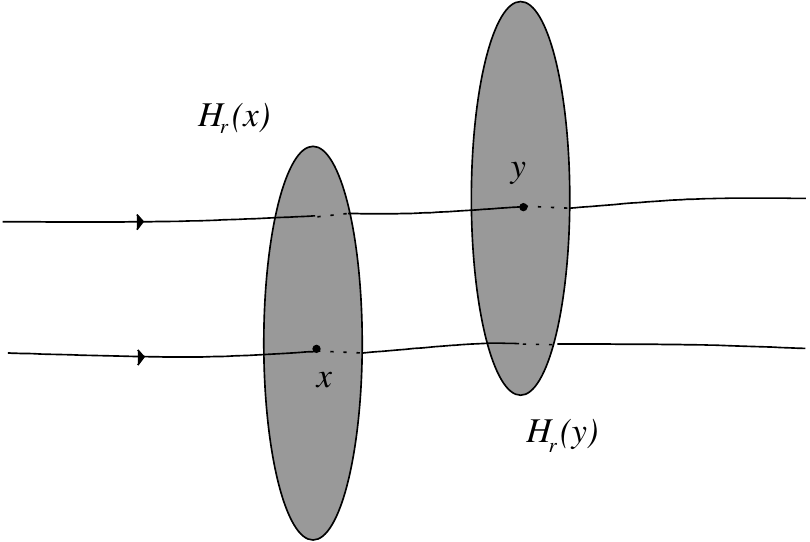}}
  \label{figSecs}
  \caption{Field of cross sections.}
\end{figure}

Then, if the flow is defined on a manifold and it admits a transverse foliation 
the construction is trivial. 
But it is known that there are smooth flows 
without a transverse foliation, see for example \cite{Goodman}. 
For these cases we have an alternative construction.
Suppose that $Y$ is the velocity field of the flow.
On a compact smooth manifold we can consider a Riemannian metric and 
define
\[
 H_r(x)=\exp_x\{v\in T_xM: v\perp Y(x), \|v\|\leq r\}
\]
where $\exp$ is the 
exponential map of the Riemannian metric and $r>0$ is small. 
This definition can be found for example in \cites{Lew, Pat}.
Unfortunately, in this setting 
items 3 and 4 above may not hold.
For continuous (non-smooth) flows on compact manifolds 
the task is even harder.
A related 
construction is given in \cite{Pat} for studying expansive flows 
of three-manifolds. 
On metric spaces local cross sections were constructed in \cite{W}.
In \cites{BW,KS,Oka,Pat} local cross sections are used 
in the study of non-smooth expansive flows on manifolds and metric spaces. 
In \cites{BW,KS,Oka} they consider a finite 
cover of flow boxes and the corresponding return maps. 
This approach is a good translating tool but 
it has some technical difficulties that makes it hard to handle. 
For example, the return maps may not be continuous at 
the boundary of the cross sections.

In this article we consider regular flows (no singular points) on compact metric spaces. 
In Section \ref{secFOCS} we develop a theory of fields of local cross sections. 
In Theorem \ref{main1} we prove that every regular flow on a compact metric space 
admits a semicontinuous, monotonous and symmetric field of cross sections. 
If in addition we have that the space is a Peano continuum, in Theorem 
\ref{secContMonSym} we conclude that
the cross sections can be assumed to vary continuously, each one being 
a connected set.
Precise definitions are given in Section \ref{secFOCS}.

As a translating tool, one may think of the cross section $H_r(x)$ as 
the ball $B_r(x)$ in the discrete time case. 
To show how this works let us translate 
the definition of wandering point. 
Recall that if $f\colon X\to X$ is a homeomorphism of a metric space 
then $x\in X$ is a \emph{wandering point} if: 
\begin{quote}
 (1) there is 
$r>0$ such that $f^n(B_r(x))\cap B_r(x)=\emptyset$ for all $n>0$. 
\end{quote}
If $\phi$ is a flow on $X$ a point $x\in X$ is said to be 
a \emph{wandering point}
if: 
\begin{quote}
(2) there are $r,t_0>0$ such that 
$\phi_t(B_r(x))\cap B_r(x)=\emptyset$ for all $t\geq t_0$. 
\end{quote}
In this definition there is a parameter $t_0$ that is not present in (1). 
Using a local cross section we have that
$x$ is a wandering point of the flow $\phi$ if and only if:
\begin{quote} (3) there is 
 $r>0$ such that $\phi_t(H_r(x))\cap H_r(x)=\emptyset$ for all $t>0$. 
\end{quote}
Notice that (3) is obtained from (1) by changing 
$f^n\mapsto \phi_t$ and $B\mapsto H$. 
This is a task that a \emph{translating machine} can do.
This machine needs to be programmed and the present paper is intended to be a contribution 
in this direction.

In Section \ref{secCWexp} we show how this machine should work in a more complex situation.
We consider 
expansive flows on Peano continua and we translate some results of 
expansive homeomorphisms.
As said in the abstract, we prove that expansive flows do not admit 
stable points. 
We show that every point contains a non-trivial continuum in its stable set 
and as a corollary we prove that no compact surface admits 
an expansive flow without singular points. 



\section{Fields of cross sections}
\label{secFOCS}

In this section we introduce the fields of compact sets, which assigns to each point a compact subset 
through the point. 
This concept is related with 
the \emph{coselections} defined
in \cite{IN}.
Fields of cross sections of a flow are a special instance 
of such fields.

In Section \ref{secCCS} we show that every regular flow 
admits a field of cross sections. For this purpose we extend the techniques of 
\cite{W} while introducing topological 1-forms. 
In Section \ref{secCFCS} we consider flows on Peano continua. 
In this case we show that every regular flow admits a continuous 
field of connected cross sections.
In Section \ref{secTFF} we introduce the transposition of fields of compact sets 
and 1-forms. These techniques are used in 
Section \ref{secMFCS} to construct monotonous fields of cross sections.
In Section \ref{secSFCS} we show that every regular flow admits a symmetric 
field of cross sections.
The main results of this section are Theorems \ref{main1} and \ref{secContMonSym}.
\subsection{Fields of compact sets}
\label{secFCS}
Let $(X,\dist)$ be a compact metric space. 
As usual, define the closed ball 
\begin{equation}
  \label{ecuBola}
  B_r(x)=\{y\in X:\dist(y,x)\leq r\}.
\end{equation}
for all $r\geq 0$ and $x\in X$. 
Also define $B_r(K)=\cup_{x\in K}B_r(x)$ if $K$ is a subset of $X$.
Denote by 
$\hyper X$ the set of compact subsets of $X$ 
equipped with the Hausdorff distance. 
If $K,L$ are compact subsets of $X$ then
the \emph{Hausdorff distance} is defined as
\[
 \dist_H(K,L)=\inf\{\epsilon>0:K\subset B_\epsilon(L),L\subset B_\epsilon(K)\}.
\]
It is known that $(\hyper X,\dist_H)$ is 
a compact metric space and a proof can be found in \cite{IN}.

\begin{df}
A \emph{field of compact sets} (or simply a \emph{field}) is a function $h\colon X\to \hyper X$ 
satisfying: 
\begin{enumerate}
  \item $x\in h(x)$ for all $x\in X$,
  \item (semicontinuity) if $x_n\to x$ and $h(x_n)\to C$ (with respect to the Hausdorff metric)
  then $C\subset h(x)$.
\end{enumerate} 
\end{df}

\begin{rmk}
  The semicontinuity property can be stated equivalently as: for all $x\in X$ and 
  $\epsilon>0$ there is $\delta>0$ such that 
  if $\dist(x,y)<\delta$ then $h(y)\subset B_\epsilon(h(x))$. 
\end{rmk}

\begin{ejp}[Two trivial fields]
Extremal examples are the \emph{null field} $x\mapsto\{x\}$
and the \emph{total field} $x\mapsto X$.
\end{ejp}

\begin{df}
 A field of compact sets $h$ is a \emph{field of neighborhoods} 
 if there is $r>0$ such that $B_r(x)\subset h(x)$ 
 for all $x\in X$.
\end{df}

\begin{ejp}[The field of closed balls]
Given $r>0$ consider the field of closed balls $B_r\colon X\to \hyper X$ 
defined by (\ref{ecuBola}).
It is a field of neighborhoods.
\end{ejp}

The following is an important example in dynamical systems.

\begin{ejp}[Local stable sets]
Given a homeomorphism $f\colon X\to X$ we
define $W^s_r,W^u_r\colon X\to \hyper X$ as 
\[
 W^s_r(x)=\{y\in X: \dist(f^n(x),f^n(y))\leq r\hbox{ for all } n\geq 0\}
\]
and 
\[
 W^u_r(x)=\{y\in X: \dist(f^n(x),f^n(y))\leq r\hbox{ for all } n\leq 0\}.
\]
We have that $W^s_r,W^u_r$ are fields of compact sets. 
Note the dependence with respect to $f$ that is omitted in the notation. 
In Section \ref{secSecFlow} this definition will be extended for flows.
\end{ejp}

We wish to remark that fields of compact sets are only assumed to be semicontinuous 
and in the next definition we require continuity. 
As before, the topology of $\hyper X$ is the one induced by the Hausdorff metric.

\begin{df}
 A continuous function $N\colon [0,1]\times X\to\hyper X$ is a 
 \emph{one-parameter field of neighborhoods}
 if:
\begin{enumerate} 
\item $N_r$ is a field of neighborhoods for all $r>0$ and
\item $N_0(x)=\{x\}$ and $N_1(x)=X$ for all $x\in X$.
\end{enumerate}
\end{df}

\begin{rmk}
  The ball operator may not be continuous and consequently $B_r$ may not define 
  a one-parameter field of neighborhoods. 
  It is not continuous with respect to $r$, for example, if $X$ is a finite set. 
  Moreover, even if $X$ is an arc $B_r$ may not be continuous. See Figure \ref{figDiscBol}.
\end{rmk}
\begin{figure}[h]
  \center{\includegraphics{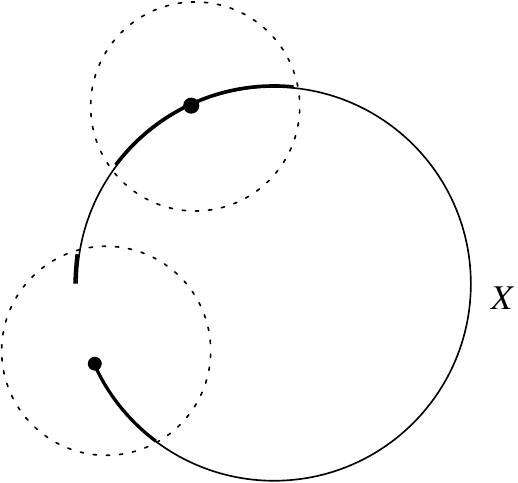}}
  \label{figDiscBol}
  \caption{If $X$ is a circular arc in the plane with the Euclidean metric, 
  $B_r(x)$ is discontinuous with respect to $x$ for some values of $r$. 
  Of course, the arc admits a metric making $B_r(x)$ continuous.}
\end{figure}

Recall that $X$ is a \emph{Peano continuum} if it is a connected, compact and locally connected metrizable space. 
A metric $\dist$ in $X$ is \emph{convex} if given $x,z\in X$, $x\neq z$, and $\alpha\in (0,\dist(x,z))$ there is 
$y\in X$ such that $\dist(x,y)=\alpha$ and $\dist(y,z)=\dist(x,z)-\alpha$. 

\begin{rmk}
\label{rmkPeanoConvex}
It is known that a compact metric space admits a convex metric defining its topology 
if and only if it is a Peano continuum. 
See  \cites{IN,Bing,Moise} for more on this subject.
\end{rmk}

Define 
\[
  \continua X=\{A\in\hyper X: A\hbox{ is connected}\}.
\]

The following result is essentially from \cites{Bing,Moise}.
\begin{thm}
\label{contB}
For a compact metric space $X$ the following statements are equivalent: 
\begin{enumerate}
  \item there is a one-parameter field of neighborhoods
   $N\colon [0,1]\times X\to\continua X$, 
  \item $X$ is a Peano continuum.
\end{enumerate}
\end{thm}

\begin{proof}
($1\to 2$). The local connection of $X$ follows because $N$ is continuous, $N_0(x)=\{x\}$ and 
each $N_r(x)$, with $r>0$, is a connected neighborhood of $x$. 
So, each point has a base of connected neighborhoods. 
It only rests to note that $X=N_1(x)\in\continua X$ to conclude that $X$ is connected.

($2\to 1$). If $X$ is a Peano continuum, then there is a convex metric $\dist$ 
defining the topology of $X$ (Remark \ref{rmkPeanoConvex}). 
Therefore, the ball operator associated to $\dist$ is a one-parameter 
field of neighborhoods (see \cite{IN} for more details).
\end{proof}

Let us give some natural definitions and remarks.

\begin{df}
Given two fields $h_1,h_2\colon X\to \hyper X$ we say that $h_1$ is a \emph{subfield} of $h_2$, denoted as $h_1\subset h_2$, 
if $h_1(x)\subset h_2(x)$ for all $x\in X$.
\end{df}

For example, $W^s_r$ and $W^u_r$ are subfields of the ball field $B_r$.

\begin{df}
 Given two fields $h_1,h_2$ define $h_1\cap h_2$ and $h_1\cup h_2$ as 
 $(h_1\cap h_2)(x)=h_1(x)\cap h_2(x)$ and 
 $(h_1\cup h_2)(x)=h_1(x)\cup h_2(x)$.
\end{df}

\begin{rmk}
\label{rmkSemicontInter}
  If $h_1, h_2\colon X\to \hyper X$ are fields then $h_1\cap h_2$ and 
  $h_1\cup h_2$ are fields.
\end{rmk}

\begin{rmk}
We have that a homeomorphism $f\colon X\to X$ is expansive if there is $r>0$ such that 
$W^s_r\cap W^u_r=O$ the null field. 
\end{rmk}

\subsection{Cross sections of a flow}
\label{secCSOAF}
Let $(X,\dist)$ be a compact metric space.

\begin{df}
\label{dfFlow}
  A \emph{flow} on $X$ is a continuous function $\phi\colon \R\times X\to X$, denoted 
  as $\phi(t,x)=\phi_t(x)$, satisfying
  $\phi_0(x)=x$ for all $x\in X$ and 
  $\phi_{s+t}(x)=\phi_s(\phi_t(x))$ for all $s,t\in \R$ and for all $x\in X$. 
  We say that $\phi$ is a \emph{regular flow} if for all $x\in X$ there is $t\in \R$ such that 
  $\phi_t(x)\neq x$, i.e., it has no equilibrium points.
\end{df}

Let $\phi\colon\R\times X\to X$ be a regular flow. 
Given a field $h\colon X\to \hyper X$ and a compact subset $I\subset \R$, $0\in I$, 
define the field $\phi_I(h)$ by $$(\phi_I(h))(x)=\{\phi_t(y):t\in I, y\in h(x)\}.$$
We say that $C\in \hyper X$ is a \emph{(local) cross section} 
through $x\in C$ if there are $\tau>0$ 
and $\gamma>0$ such that $B_\gamma(x)\subset \phi_{[-\tau,\tau]}(C)$ and 
 $C\cap\phi_{[-\tau,\tau]}(y)=\{y\}$ for all $y\in C$.

\begin{df}
\label{dfFCrossSec}
A field $H\colon X\to \hyper X$ is a \emph{field of cross sections} 
if there is $\tauu >0$ such that $\phi_{[-\tauu ,\tauu ]}(H)$ is a field of neighborhoods and 
$H(x)\cap\phi_{[-\tauu ,\tauu ]}(y)=\{y\}$ if $x\in X$ and $y\in H(x)$. 
\end{df}

\begin{prop}
  If $H$ is a field of cross sections then there is $\tau>0$ such that
  $\phi\colon[-\tau,\tau]\times H(x)\to X$ is injective, for all $x\in X$, and consequently 
  it is a homeomorphism onto its image.
\end{prop}

\begin{proof}
  Take $\tauu >0$ from Definition \ref{dfFCrossSec}. Since $\phi$ is a regular flow there is 
  $\tau\in (0,\tauu /2)$ such that $\phi_t(z)\neq z$ if $z\in X$ and $0<|t|\leq 2\tau$.
  Suppose that $\phi_t(y)=\phi_s(z)$ with $|s|,|t|\leq \tau$ and $y,z\in H(x)$. 
  Then $\phi_{t-s}(y)=z$ and $|t-s|\leq 2\tau<\tauu $.
  Since $H(x)\cap\phi_{[-\tauu ,\tauu ]}(y)=\{y\}$ we conclude that $y=z$. 
  Then $\phi_{t-s}(y)=y$ with $|t-s|\leq 2\tau$. This implies that $t=s$. 
  This proves the injectivity of the considered restriction of $\phi$. 
  The continuity of the inverse follows by the continuity of $\phi$ and the compactness of $[-\tau,\tau]\times H(x)$.
\end{proof}

\begin{df}
If $H$ is a field of cross sections, $\tau>0$ and 
$\phi\colon[-\tau,\tau]\times H(x)\to X$ is injective, for all $x\in X$ 
then we say that $H$ is a field of cross sections of \emph{time} $\tau$.
\end{df}

From the definitions we are assuming that our fields $H$ are semicontinuous. 
Then, it can be the case that $x_n\to x$, $r>0$, $y\in H(x)$ 
and $B_r(y)\cap H(x_n)=\emptyset$ for all $n\geq 1$. 
In this case we could say that $y$ is \emph{discontinuity point} 
(a very confusing terminology that will not 
be used in the sequel).
The following lemma means that if $H$ is a field of cross sections 
then the discontinuity points in $H(x)$ are uniformly far from $x$.

\begin{lem}
\label{cuasicont}
  If $H$ is a field of cross sections of time $\tau$ and 
  $B_\gamma\subset \phi_{[-\tau,\tau]}(H)$ 
  then for all $\epsilon\in(0,\tau]$ and for all $\gamma'\in (0,\gamma)$ 
  there is $\delta>0$ such that 
  if $\dist(x,y)\leq\delta$, 
  then $H(y)\cap B_{\gamma'}(y)\subset\phi_{[-\epsilon,\epsilon]}(H(x))$.
\end{lem}

\begin{proof}
  Arguing by contradiction, assume that there are $\epsilon>0$, $\gamma'\in (0,\gamma)$, 
  $x_n,y_n\to p$ and $z_n\in H(y_n)\cap B_{\gamma'}(y_n)$ such that 
  \begin{equation}
    \label{ecuznH}
    z_n\notin \phi_{[-\epsilon,\epsilon]}(H(x_n)). 
  \end{equation}
  Since $\dist(x_n,y_n)\to 0$ we can suppose that $\dist(x_n,y_n)<\gamma-\gamma'$. 
  In this way, $B_{\gamma'}(y_n)\subset B_\gamma(x_n)$.
  Then $z_n\in B_\gamma(x_n)$. 
  Since $B_\gamma\subset \phi_{[-\tau,\tau]}(H)$, 
  we have that $z_n\in \phi_{[-\tau,\tau]}(H(x_n))$ and consequently 
  there is $t_n\in{[-\tau_n,\tau_n]}$ such that $\phi_{t_n}(z_n)\in H(x_n)$. 
  By (\ref{ecuznH}) we have that $|t_n|>\epsilon$. 
  Now we take limit as $n\to\infty$. 
  Assume that $t_n\to \tu $ with $|\tu |\in [\epsilon,\tau]$, $z_n\to z$ an recall that $x_n,y_n\to p$. 
  Since $H$ is semicontinuous and $z_n\in H(y_n)$ we have that $z\in H(p)$. 
  Also, $\phi_{\tu }(z)\in H(p)$. Since $|t|\leq\tau$ and $t\neq 0$ we have a contradiction 
  with the definition of cross section and the lemma is proved.
\end{proof}

\begin{prop}
\label{cuasicont2}
If $H$ is a field of cross sections then 
$\phi_{[-\epsilon,\epsilon]}(H)$ is 
a field of neighborhoods for all $\epsilon>0$.
\end{prop}

\begin{proof}
Given $\epsilon>0$ take any $\gamma'\in (0,\gamma)$ and consider 
$\delta>0$ from Lemma \ref{cuasicont}. 
Then, if $\dist(x,y)\leq\delta$ we have, by the lemma, that $y\in \phi_{[-\epsilon,\epsilon]}(H(x))$. 
Therefore, $B_\delta\subset \phi_{[-\epsilon,\epsilon]}(H)$ and 
$\phi_{[-\epsilon,\epsilon]}(H)$ is a field of neighborhoods for all $\epsilon>0$. 
\end{proof}

\begin{prop}
\label{cajaDos}
If $H$ is a field of cross sections 
then $N_\epsilon(x)=\cup_{|t|\leq\epsilon}H(\phi_t(x))$ is a 
field of neighborhoods for all $\epsilon>0$.
\end{prop}

\begin{proof}
By contradiction assume that no $B_\rho$ is contained in $N_\epsilon$ 
for some fixed $\epsilon>0$. 
Then we can take $x_n,y_n$ such that $\dist(x_n,y_n)\to 0$ 
and $y_n\notin N_\epsilon(x_n)$ for all $n\geq 1$. 
Take $\delta>0$ (from Proposition \ref{cuasicont2}) such that $B_\delta\subset \phi_{[-\epsilon,\epsilon]}(H)$. 
Also assume that $y_n\in B_\delta(x_n)$ for all $n\geq 1$. 
Suppose that $H$ is a field of cross sections of time $\tau$.
If $\epsilon$ and $\delta$ are small we can assume that 
$B_\delta(x_n)\subset \phi_{[-\tau,\tau]}(H(\phi_t(x_n)))$ if $|t|\leq\epsilon$.
Then, for each $n$, there are $t_n,s_n$ such that $z_n=\phi_{s_n}(y_n)\in H(\phi_{-\epsilon}(x_n))$ 
and $\phi_{s_n}(y_n)\in H(x_n)$. 
We can assume that $0<s_n\leq t_n$. 
By Lemma \ref{cuasicont} we have that $t_n\to 0$. 
Then $s_n\to 0$. 
Taking limit $n\to\infty$ we obtain: $\lim z_n=\phi_\epsilon(\lim x_n)$ 
and $\lim z_n\in H(\phi_{-\epsilon}(\lim x_n))$ which is a contradiction.
\end{proof}

\begin{prop}
\label{propSubCross}
  If $N$ is a field of neighborhoods and $H$ is a field of cross sections 
  then $N\cap H$ is a field of cross sections.
\end{prop}

\begin{proof}
  The semicontinuity of $N\cap H$ follows by Remark \ref{rmkSemicontInter}. 
  Let $\tau>0$ be a time for $H$.
  Take $\gamma>0$ such that $B_\gamma\subset N\cap \phi_{[-\tau,\tau]}(H)$.
  Consider $\epsilon>0$ such that $\diam(\phi_{[0,\epsilon]}(z))<\gamma/2$ for all $z\in X$.
  By Proposition \ref{cuasicont2} there is $\delta\in(0,\gamma/2)$ such that $B_\delta(x)\subset \phi_{[-\epsilon,\epsilon]}(H(x))$.
  
  Let us show that $B_\delta\subset \phi_{[-\tau,\tau]}(N\cap H)$.
  For all $y\in B_\delta(x)$ there is $t\in [-\tau,\tau]$ such that 
  $z=\phi_t(y)\in H(x)$. 
  Since $\dist(y,z)<\gamma/2$ and $\dist(x,y)\leq\delta<\gamma/2$ we have that $z\in B_\gamma(x)\subset N(x)$. 
  Then $z\in N(x)\cap H(x)$. 
  Since $y =\phi_{-t}(z)$ and $|t|\leq\tau$ we have that $y\in \phi_{[-\tau,\tau]}(N(x)\cap H(x))$. 
  Then $B_\delta(x)\subset \phi_{[-\tau,\tau]}(N(x)\cap H(x))$ and the proof ends.
\end{proof}

For future reference we state the following result.

\begin{prop}
\label{H123}
  If $H_1\subset H_2\subset H_3$, $H_1,H_3$ are fields of cross sections 
  and $H_2$ is semicontinuous then 
  $H_2$ is a field of cross sections.
\end{prop}

\begin{proof}
  It is direct from the definitions.
\end{proof}

\subsection{Constructing cross sections}
\label{secCCS}

Given a field of neighborhoods $N\colon X\to\hyper X$ define 
$$U(N)=\{(x,y)\in X\times X:y\in N(x)\}.$$ 

\begin{df}[Topological forms]
  A \emph{1-form} is a continuous function 
  $$\omega\colon U(N)\to \R$$ such that $\omega_x(x)=0$.
  Define the field of compact sets 
  $\ker(\omega)\colon X\to\hyper X$ by 
  $$\ker(\omega)_x=\{y\in N(x):\omega_x(y)=0\}.$$
\end{df}

Given a flow $\phi$ on $X$ and a 1-form 
$\omega$ define 
\[
  \dot\omega_x(y)=\lim_{t\to 0}\frac{\omega_x(\phi_t(y))-\omega_x(y)}t
\]
if this limit exists. 

\begin{rmk}
  It could be useful to think of $U(N)$ as Milnor's tangent microbundle 
  as defined in \cite{Milnor}.
\end{rmk}

The following results are based in \cite{W}*{Section 29}. 
They are stated here in such a way that they can be used 
and extended later.

\begin{prop}
\label{lemFunD}
If $\omega$ is a 1-form with $\dot\omega$ continuous and non-vanishing then 
there is $\radio_1>0$ 
such that 
$H_\rho=B_\rho\cap \ker(\omega)$ 
is a field of cross sections for all $\rho\in (0,\rho_1)$.
\end{prop}

\begin{proof}
Since $\omega_x$ is continuous we have that $H_\rho(x)$ is a compact set and 
since $\omega_x(x)=0$ we have that $x\in H_\rho(x)$.
Let $a,\rho_1>0$ be such that $B_{\rho_1}\subset N$ and 
$|\dot\omega_x(y)|\geq a$ if $y\in B_{\rho_1}(x)$.
Given $\radio\in (0,\rho_1)$ there is $\tau>0$ such that if $\dist(x,y)<\radio$ then 
$\phi_t(y)\in B_{\rho_1}(x)$ for all $t\in[-\tau,\tau]$. 
Consider $\gamma>0$ such that if $\dist(x,y)\leq\gamma$ 
then $|\omega_x(y)|\leq a\tau$. 
To show that $B_\gamma(x)\subset \phi_{[-\tau,\tau]}H_\radio (x)$ 
take $y\in B_\gamma(x)$. 
Since $|\dot \omega_x(y)|\geq a$ there is $s\in[-\tau,\tau]$ such that 
$\omega_x(\phi_s(y))=\omega_x(x)$. Therefore $\phi_s(y)\in H_\radio(x)$. 
Finally, we have that $H_\radio(x)\cap \phi_{[-\tau,\tau]}(y)=\{y\}$ 
because $\dot \omega_x(y)\neq 0$ for all $y\in B_{\rho_1}(x)$.

To show that $H_\rho$ is semicontinuous 
take $y_n\in H_\rho(x_n)$, $x_n\to x$, and $y_n\to y$. 
We will show that $y\in H_\radio(x)$. 
Since $\omega$ is continuous we see that $\omega_x(y)=\omega_x(x)$. 
Since $\dist(x_n,y_n)\leq \radio_n$ for all $n\geq 1$ we have that $\dist(x,y)\leq\radio$. 
Therefore $y\in H_\radio(x)$. 
\end{proof}

\begin{prop}
\label{constrSec}
If $v\colon U(N)\to\R$ is a 1-form and
there is $\tu >0$ such that if $t\in[0,\tu ]$ then 
$(x,\phi_t(x))\in U(N)$ for all $x\in X$ and 
$v_x(\phi_{\tu }(x))\neq 0$ for all $x\in X$ then
\begin{equation}
\label{ecuSecT}
\omega_x( y)=
\int_{0}^{\tu }v_x(\phi_s(y))ds-
\int_{0}^{\tu }v_x(\phi_s(x))ds
\end{equation}
is a 1-form with 
\[
\dot\omega_x(y)= v_x(\phi_{\tu }(y))-v_x(y).
\]
Consequently, 
$\dot\omega_x(y)\neq 0$ if $y$ is close to $x$ and 
$\dot\omega$ is continuous.
\end{prop}

\begin{proof}
Notice that $\omega_x(x)=0$ for all $x\in X$.
Take $\radio_0>0$ such that if $\dist(x,y)\leq\radio_0$ and $t\in[0,\tu ]$ then
$(x,\phi_t(y))\in U$. 
In this way (\ref{ecuSecT}) is well defined if $\dist(x,y)\leq\rho_0$.
Let $\rho_1\in(0,\rho_0)$ and $a>0$ be such that 
$v_x(\phi_{\tu }(y))-v_x(y)>a$ for all $x\in X$ and $y\in B_{\rho_1}(x)$.
Notice that 
 \[
 \begin{array}{ll}
  \omega_x(\phi_t(y))-\omega_x(y)&= 
  \int_0^{\tu }v_x(\phi_s(\phi_t(y)))ds-\int_0^{\tu }v_x(\phi_s(y))ds\\
  &=\int_t^{\tu +t}v_x(\phi_s(y))ds - \int _0^{\tu }v_x(\phi_s(y)) ds\\
  &=\int_{\tu }^{\tu +t} v_x(\phi_s(y)) ds - \int_0^tv_x(\phi_s(y)) ds.
 \end{array}
 \]
Then 
$ \dot \omega_x(y)= v_x(\phi_{\tu }(y))-v_x(y)\geq a$ if $y\in B_{\rho_1}(x)$.
\end{proof}

\begin{thm}
\label{teoW}
 If $X$ is a compact metric space then every regular flow admits 
 a field of cross sections.
\end{thm}
\begin{proof}
If $\phi$ is a regular flow  then there is $\tu >0$ such that 
$\phi_{\tu }(x)\neq x$ for all $x\in X$.
Define $v_x(y)=\dist(x,y)$ for all $x,y\in X$.
The result follows by Propositions \ref{lemFunD} and \ref{constrSec}.
\end{proof}

\subsection{Continuous fields of connected cross sections}
\label{secCFCS}

In this section we will obtain continuous fields 
of connected cross sections assuming that $X$ is a Peano continuum.
For this, we will need the flow projection from a flow box to a cross section.

\begin{df}
  If $H$ is a field of cross sections of time $\tau$ define the 
  \emph{field of flow boxes} associated to $H$ as $F\colon X\to \hyper X$ by 
  $F(x)=\phi_{[-\tau,\tau]}(H(x))$.
\end{df}

\begin{rmk}
  By Proposition \ref{cuasicont2} every field of flow boxes is a field of neighborhoods.
\end{rmk}

\begin{df}
  Given the field of flow boxes $F$ associated to $H$ define the \emph{flow projection} 
  $\pi_x\colon F(x)\to H(x)$ by $\pi_x(y)=\phi_t(y)\in H(x)$ with $|t|\leq\tau$. 
  Given a field of compact sets $M\subset F$ 
  define $\pi(M)$ by $(\pi(M))(x)=\pi_x(M(x))$.
\end{df}

\begin{prop}
\label{piN}
  If $H$ is a field of cross sections with $F$ the associated field of flow boxes and 
  $N\subset F$ is a field of compact sets then: 
  \begin{enumerate}
    \item $\pi(N)$ is a field of compact sets,
    \item if $N(x)$ is connected then $(\pi(N))(x)$ is connected,
    \item if $N$ is continuous then $\pi(N)$ is continuous and
    \item if $N$ is a field of neighborhoods then $\pi(N)$ is a field of cross sections.
  \end{enumerate}
\end{prop}

\begin{proof} Define $J=\pi(N)$.

  Item 1. We have to prove that $J$ is semicontinuous. 
  Take $x_n\to x$ and $y_n\in J(x_n)$ with $y_n\to y$. 
  Consider $z_n\in N(x_n)$ with $\pi_{x_n}(z_n)=\phi_{t_n}(z_n)=y_n$, $z_n\to z$ 
  and $t_n\to t$.
  Since $N$ is semicontinuous we have that $z\in N(x)$. 
  Also, $\phi_t(z)=y$, $t\in [-\tau,\tau]$ and since $y_n\in J(x_n)\subset H(x_n)$  
  and $H$ is semicontinuous we have that $y\in H(x)$. 
  Then $y=\pi_x(z)$ and $y\in J(x)$.
 
  Item 2. It follows because $\pi_x$ is continuous.
  
  Item 3. We know that $J$ is semicontinuous. Assume that it is not continuous. 
  Then there are $\epsilon>0$ and $x,x_n\in X$ such that 
  $x_n\to x$, $J(x)\nsubseteq B_\epsilon(J(x_n))$. 
  Then there is $z_n\in J(x)$ such that $z_n\notin B_\epsilon(J(x_n))$. 
  Take $u_n\in N(x)$ and $t_n\in[-\tau,\tau]$ such that $\phi_{t_n}(u_n)=z_n$. 
  Since $N$ is continuous there is $v_n\in N(x_n)$ such that $\dist(u_n,v_n)\to 0$. 
  Take $s_n\in[-\tau,\tau]$ such that $z'_n=\phi_{s_n}(v_n)\in J(x_n)$. 
  We have that $\dist(z'_n,z_n)>\epsilon$. 
  Taking limit (and subsequences) $n\to \infty$ we can assume that $t_n\to t$, $s_n\to s$, $z_n\to z$, $z'_n\to z'$ and $u_n,v_n\to u$.
  By the semicontinuity of $J$ we have that $z,z'\in J(x)$. Also $\dist(z,z')\geq \epsilon$. 
  By continuity we have $\phi_t(u)=z$ and $\phi_s(v)=z'$. 
  Therefore, $z=\phi_{t-s}(z')$. 
  Since $z,z'\in J(x)$ and $|t-s|\leq 2\tau$ we have a contradiction because $\tau$ is a time for $H$.
  
  Item 4. By Proposition \ref{propSubCross} we know that $N\cap F$ is a field of cross sections. 
  Notice that $N\cap F\subset J\subset H$. 
  Since we have proved that $J$ is semicontinuous we can apply 
  Proposition \ref{H123} to conclude that $J$ is a field of cross sections.
\end{proof}

The following result gives us continuous fields of connected cross sections.

\begin{prop}
\label{subfcont}
 If $X$ is a Peano continuum and $\phi$ is a regular flow 
 then every field of cross sections $H\colon X\to \hyper X$ 
 admits a continuous subfield of cross sections $H'\colon X\to\continua X$.
\end{prop}

\begin{proof}
  Let $H$ be a field of cross sections. 
  Since $X$ is a Peano continuum we know from Theorem \ref{contB} that 
  there is a continuous field of neighborhoods $N\colon X\to\continua X$ with 
  $N\subset F$.
  By Proposition \ref{piN} we know that $H'=\pi(N)\colon X\to\continua X$ is a
  continuous field of cross sections satisfying $H'\subset H$.
\end{proof}

\begin{thm}
 If $X$ is a Peano continuum then every regular flow 
 admits a continuous field of connected cross sections $H\colon X\to\continua X$.
\end{thm}

\begin{proof}
  By Theorem \ref{teoW} we know that there is a field of cross sections $H'$. 
  By Theorem \ref{contB} there is a continuous field of neighborhoods 
  $N\subset \phi_{[-\tau,\tau]}(H')$ 
  with $N\colon X\to \continua X$. 
  By Proposition \ref{piN} we have that $H=\pi(N)$ is a continuous field of cross sections, moreover, 
  each $H(x)$ is connected.
\end{proof}

\subsection{Transpose fields and 1-forms}
\label{secTFF}

The formalism introduced in this section will be applied in the next section 
to study the monoticy of fields of cross section.

\begin{df}[Transpose field]
  Given a field $h\colon X\to\hyper X$ define its \emph{transpose} 
  $h^\tra\colon X\to\hyper X$ by
  $$h^\tra(x)=\{y\in X:x\in h(y)\}.$$ 
  We say that $h$ is \emph{symmetric} if $h^\tra=h$.
\end{df}

\begin{prop}
\label{propsAdj}
  The following statements hold:
  \begin{enumerate}
    \item $h^{\tra\tra}=h$ for every field $h$,
    \item $h^\tra$ is a field of compact sets if and only if $h$ is a field of compact sets,
    \item the field of balls $B_\rho$ is symmetric,
    \item $h_1\subset h_2$ if and only if $h_1^\tra\subset h_2^\tra$,
    \item $N$ is a field of neighborhoods if and only if $N^\tra$ is a field of neighborhoods,
    \item $(h_1\cap h_2)^\tra=h_1^\tra\cap h_2^\tra$
  \end{enumerate}
\end{prop}

\begin{proof}
\begin{enumerate}
  \item Notice that: $y\in h(x)\Leftrightarrow x\in h^\tra(y)\Leftrightarrow 
  y\in h^{\tra\tra}(x)$.
  \item Assume that $h$ is semicontinuous and take 
  $y_n\in h^\tra(x_n)$ with $x_n\to x$ and $y_n\to y$. 
  Then $x_n\in h(y_n)$. Since $h$ is semicontinuous we have that 
  $x\in h(y)$. Then $y\in h^\tra(x)$. The converse follows by this and item 1.
  \item It follows by the symmetry of the metric. 
  \item Suppose that $h_1\subset h_2$. 
  If $y\in h_1^\tra(x)$ then $x\in h_1(y)$. 
  So, $x\in h_2(y)$ and $y\in h_2^\tra(x)$.
  The converse follows by this and item 1.
  \item If $N$ is a field of neighborhoods then there is $\rho>0$ such that 
  $B_\rho\subset N$. Then (item 4) $B^\tra_\rho\subset N^\tra$. 
  Since $B_\rho=B^\tra_\rho$ (item 3)
  we have that $B_\rho\subset N^\tra$.
  \item Notice that $y\in (h_1\cap h_2)^\tra_x\Leftrightarrow
  x\in(h_1\cap h_2)_y
  \Leftrightarrow [x\in h_1(y)$ and $x\in h_2(y)]
  \Leftrightarrow
  [y\in h_1^\tra(x)$ and $y\in h_2^\tra(x)]\Leftrightarrow y\in (h_1^\tra\cap h_2^\tra)_x$.\qedhere
\end{enumerate}
\end{proof}

\begin{df}
  Given a 1-form $\omega\colon U(N)\to\R$ define its 
  \emph{transpose 1-form} $$\omega^\tra\colon U(N^\tra)\to\R$$ as 
  $\omega^\tra_x(y)=\omega_y(x)$. We say that $\omega$ is \emph{symmetric} 
  if $\omega^\tra=\omega$.
\end{df}

\begin{prop}
\label{kerOmega}
It holds that
     $\ker(\omega^\tra)=(\ker(\omega))^\tra$.
\end{prop}

\begin{proof}
It follows by definitions:
    $y\in\ker(\omega^\tra)_x\Leftrightarrow
    \omega^\tra_x(y)=0\Leftrightarrow
    \omega_y(x)=0\Leftrightarrow
    x\in\ker(\omega)_y\Leftrightarrow
    y\in(\ker(\omega))_x^\tra$.
\end{proof}

\subsection{Monotonous fields of cross sections}
\label{secMFCS}

\begin{df}
A field of cross sections $H\colon X\to \hyper X$ is \emph{monotonous} 
if there is $\epsilon>0$ such that for all $t\in (0,\epsilon)$ it holds 
that $H(x)\cap H(\phi_t(x))=\emptyset$ for all $x\in X$.
\end{df}

\begin{rmk}
\label{subMono}
  Every subfield of cross sections of a monotonous field of cross sections is monotonous.
\end{rmk}

\begin{prop}
\label{monoSecEquiv}
  A field of cross sections $H$ is monotonous 
  if and only if $H^\tra$ is a monotonous field of cross sections.
\end{prop}

\begin{proof}
First assume that $H$ is monotonous.
By Proposition \ref{propsAdj} we know that $H^\tra$ is a field of compact sets. 
By Proposition \ref{cajaDos} we know that 
$$N_\epsilon(y)=\cup_{|t|\leq\epsilon}H(\phi_t(y))$$ is 
a field of neighborhoods 
for all $\epsilon>0$.
For $\epsilon>0$ fixed we can take $\rho>0$ such that $B_\rho\subset N_\epsilon$. 
Let us prove that $B_\rho\subset \phi_{[-\epsilon,\epsilon]}(H^\tra)$. 
Take $y\in B_\rho(x)$. Then $x\in N_\epsilon(y)$. 
Therefore there is $s\in[-\epsilon,\epsilon]$ such that 
$x\in H(\phi_s(y))$. 
Thus, $\phi_s(y)\in H^\tra(x)$ and $y\in \phi_{[-\epsilon,\epsilon]}(H^\tra(x))$. 

Since $H$ is monotonous there is $\epsilon>0$ such that 
$H(x)\cap H(\phi_t(x))=\emptyset$ 
for all $x\in X$ and $t\in[-\epsilon,\epsilon]$, $t\neq 0$.
Let us show that $H^\tra(x)\cap \phi_{[-\epsilon,\epsilon]}(y)=\{y\}$ for all 
$x\in X$ and $y\in H^\tra(x)$. 
Take $z\in H^\tra(x)\cap \phi_{[-\epsilon,\epsilon]}(y)$. 
Then $x\in H(z)$ and there is $s\in[-\epsilon,\epsilon]$ such that $z=\phi_s(y)$ with 
$y\in H^\tra(x)$, i.e., $x\in H(y)$. 
Then $x\in H(z)\cap H(\phi_s(z))$. Since $H$ is monotonous, this implies that 
$s=0$ and $y=z$ as we wanted to prove.

Since $H$ is a field of cross sections there is $\tau>0$ such that 
$H(y)\cap \phi_{[-\tau,\tau]}(x)=\{x\}$ for all 
$y\in X$ and $x\in H(y)$. 
We will prove that 
$H^\tra(x)\cap H^\tra(\phi_t(x))=\emptyset$ if 
$0<|t|\leq\tau$. 
If $y\in H^\tra(x)\cap H^\tra(\phi_t(x))$ with $|t|\leq\tau$ 
then $x\in H(y)$ and $\phi_t(x)\in H(y)$. 
This implies that $t=0$ and the first part of the proof ends.

The converse follows because $H^{\tra\tra}=H$.
\end{proof}

The notation $\dot v^\tra$ means the derivative of $v^\tra$ (it is not the transpose of $\dot v$).

\begin{prop}
\label{secvvest}
If $v$ is a 1-form such that $\dot v\neq 0$ and $\dot v^\tra\neq 0$ 
then $H_\rho=B_\rho\cap\ker(v)$ 
is a monotonous field of cross sections for $\rho$ small.
\end{prop}

\begin{proof}
  By Proposition \ref{lemFunD} we know that 
  $H_\rho$ is a field of cross sections because $\dot v\neq 0$. 
  Now suppose that $y\in H(x)\cap H(\phi_t(x))$. 
  This implies that $v_x(y)=v_{\phi_t(x)}(y)=0$. 
  Then $v^\tra_y(x)=v^\tra_y(\phi_t(x))=0$. 
  Since $\dot v^\tra\neq 0$ we have that $t=0$ and $H_\rho$ is monotonous.
\end{proof}

\begin{prop}
\label{secmonot}
  If $\dot\omega=1$ and the 1-form $v$ is defined by
  $$v_x(y)=\int_0^{\tu }w^\tra_x(\phi_s(y)) ds-\frac{\tu ^2}2$$
  then $\dot v^\tra=\tu$, 
  $\dot v_y(x)=\omega^\tra_y(\phi_{\tu}(x))-\omega^\tra_y(x)$ and $H_\rho=B_\rho\cap\ker(v)$ is a monotonous field of cross sections for $\rho$ small.
\end{prop}

\begin{proof}
Since $\dot\omega=1$ we have that $\omega_x(\phi_t(y))=t+\omega_x(y)$. 
Then
\[
  \begin{array}{ll}
    v^\tra_x(\phi_t(y))-v^\tra_x(y)
    &= v_{\phi_t(y)}(x)-v_y(x)\\
    &=\int_0^{\tu}[\omega^\tra_{\phi_t(y)}(\phi_s(x))-\omega^\tra_y(\phi_s(x))]ds\\
    &=\int_0^{\tu}[\omega_{\phi_s(x)}(\phi_t(y))-\omega_{\phi_s(x)}(y)]ds\\
    &=\int_0^{\tu}tds=t\tu.
  \end{array}
\]
Therefore $\dot v^\tra=\tu$.
We have that $v_x(x)=0$ for all $x\in X$ because $\omega_x(x)=0$ and: 
\[
  \begin{array}{ll}
    \int_0^{\tu }w^\tra_x(\phi_s(x)) ds
    &=\int_0^{\tu }w_{\phi_s(x)}(x) ds\\
    &=\int_0^{\tu}[w_x(x) +s]ds\\
    &={\tu}^2/2.
  \end{array}
\]
Notice that 
\[
H^\tra_\rho
  =[B_\rho\cap\ker(v)]^\tra
  =B^\tra_\rho\cap[\ker(v)]^\tra
  =B_\rho\cap\ker(v^\tra)
\]
By Proposition \ref{constrSec}
we have that $\dot v_y(x)=\omega^\tra_y(\phi_{\tu}(x))-\omega^\tra_y(x)$.
Then $$\dot v_x(x)=\omega^\tra_x(\phi_{\tu}(x))=\omega_{\phi_{\tu}(x)}(x).$$ 
Define $y=\phi_{\tu}(x)$. 
Then 
$$\dot v_x(x)= \omega_y(\phi_{-\tu}(y))=-\tu.$$
Now the result follows by Proposition \ref{secvvest}.
\end{proof}

\begin{prop}
\label{omegaPuntoUno}
  Every regular flow admits a 1-form $\omega$ with $\dot\omega=1$.
\end{prop}

\begin{proof}
  By Theorem \ref{teoW} we know that there is a field of cross sections $H'$. 
Let $F=\phi_{[-\tau,\tau]}(H')$ be the associated field of flow boxes. 
Define $\omega_x\colon F(x)\to [-\tau,\tau]$ 
by 
$$
 \phi_{-\omega_x(y)}(y)\in H'(x)
$$
for $y\in F(x)$. Notice that 
\begin{equation}
 \label{ecuTime}
\omega_z(\phi_t(x))=t+\omega_z(x)
\end{equation}
if $|t+\omega_z(x)|\leq\tau$.
Then $\dot\omega=1$.
\end{proof}

\begin{thm}
\label{monoSec}
If $X$ is a compact metric space then every regular flow on $X$
admits a monotonous field of cross sections.
\end{thm}

\begin{proof}
It follows by Propositions \ref{secmonot} and \ref{omegaPuntoUno}.
\end{proof}

\begin{cor}
\label{secPeano}
If $X$ is a Peano continuum then every regular flow on $X$ admits 
a continuous and monotonous field of cross sections.
\end{cor}

\begin{proof}
 By Theorem \ref{monoSec} we know that there is a monotonous 
 field of cross sections $H$. 
 By Proposition \ref{subfcont} there is a continuous field of cross sections 
 $H'\subset H$ with $H'\colon X\to \continua X$. 
 By Remark \ref{subMono} we have that $H'$ is monotonous.
\end{proof}

\subsection{Symmetric fields of cross sections}
\label{secSFCS}

\begin{df}
  A 1-form $\omega$ is \emph{anti-symmetric} if $\omega^\tra=-\omega$.
\end{df}

\begin{prop}
\label{kerOmega2}
  If $\omega$ is anti-symmetric then $\ker(\omega)$ is a 
  symmetric field of compact sets.
\end{prop}

\begin{proof}
  It is a direct consequence of the definitions. 
  It could also be derived from Proposition \ref{kerOmega}.
\end{proof}

\begin{prop}
\label{PropMonSymCross}
  If $\omega$ is anti-symmetric and $\dot\omega\neq 0$ then $H=B_\rho\cap\ker\omega$ 
  is a monotonous and symmetric field of cross sections if $\rho$ is small.
\end{prop}

\begin{proof}
Since $\omega$ is symmetric we know by Proposition \ref{kerOmega2} that 
$\ker(\omega)$ is a symmetric field of compact sets. 
By Proposition \ref{propsAdj} we conclude that $H$ is a symmetric 
field of compact sets. 
Since $\omega$ is anti-symmetric we have that $\dot\omega^\tra=-\dot\omega$.
Then, we can apply Proposition \ref{secvvest} to conclude that 
$H$ is a monotonous field of cross sections. 
\end{proof}

\begin{prop}
\label{antiSymForm}
Every regular flow admits an anti-symmetric 1-form $\omega$ with $\dot\omega\neq 0$.
\end{prop}

\begin{proof}
By Propositions \ref{omegaPuntoUno} and \ref{secmonot} 
there is $v$ such that $\dot v>0$ and $\dot v^\tra<0$. 
Define $\omega=v-v^\tra$. 
In this way $\omega^\tra=-\omega$ and $\dot\omega>0$.
\end{proof}

\begin{thm}
\label{main1}
If $X$ is a compact metric space then every regular flow on $X$
admits a symmetric and monotonous field of cross sections.
\end{thm}

\begin{proof}
It follows by Propositions \ref{antiSymForm}
and \ref{PropMonSymCross}.
\end{proof}

\begin{df}
  We say that a field of compact sets $H$ 
  is \emph{locally symmetric} if there is $\delta>0$ such that 
  $B_\delta\cap H$ is symmetric.
\end{df}

\begin{thm}
\label{secContMonSym}
If $X$ is a Peano continuum and $\phi$ is a regular flow on $X$ then 
there are
a symmetric and monotonous field of cross sections $H'$ and
a continuous one-parameter field
$H\colon [0,r]\times X\to \continua X$ 
satisfying:
\begin{enumerate}
\item $H_\epsilon\colon X\to \continua X$ is 
  a monotonous, continuous and locally symmetric 
  field of connected cross sections for all $\epsilon\in (0,r]$, ,
\item $H_\epsilon\subset H'$ for all $\epsilon\in[0,r]$
\item $H_0(x)=\{x\}$ for all $x\in X$,
\item if $0\leq\epsilon\leq\epsilon'\leq r$ then $H_\epsilon\subset H_{\epsilon'}$,
\end{enumerate}
\end{thm}

\begin{proof}
The field of cross sections $H'$ is given by Theorem \ref{main1}.
Let $F$ be the field of flow boxes associated with $H'$ en denote 
by $\pi$ its flow projection. 
Consider a metric in $X$ such that the field of balls $B_\epsilon$ is continuous.
Take $r>0$ such that $B_r\subset F$. 
Define $H_\epsilon=\pi(B_\epsilon)$ for $\epsilon\in [0,r]$.
The result now is direct.
\end{proof}

\section{Expansive flows}
\label{secCWexp}

In this section we will apply our constructions to the study of expansive flows. 
In Section \ref{secSecFlow} we define the sectional flow, which is similar to 
the Poincaré return map (or the linear flow). 
It is also some kind of holonomy. 
The definition of the sectional flow does not depend on the expansivity of the flow.
In Sections \ref{secCwexpFaS} and \ref{secStables} 
we show our results on expansive flows.

\subsection{Sectional flow}
\label{secSecFlow}
Let $(X,\dist)$ be a Peano continuum and denote by $\phi$ a regular 
flow on $X$. 
Consider a continuous, monotonous and locally symmetric 
one-parameter field of connected cross sections 
$H_\epsilon\colon X\to\continua X$ given by Theorem \ref{secContMonSym}.

\begin{lem}
\label{lemUniRep}
 Given $x,y\in X$ and a connected set $I\subset \R$ with $0\in I$ there is at most one 
 continuous function $h\colon I\to \R$ such that $h(0)=0$ and 
 $\phi_{h(t)}(y)\in H(\phi_t(x))$ for all $t\in I$.
 In this case $h$ is strictly increasing.
\end{lem}

\begin{proof}
Suppose that $h_1,h_2\colon I\to \R$ are continuous, $h_1(0)=h_2(0)=0$ 
and for all $t\in I$ it holds that
$\phi_{h_i(t)}(y)\in H(\phi_t(x))$, for $i=1,2$.
Define $J=\{t\in I:h_1(t)=h_2(t)\}$. 
We have that $J$ is closed because $h_1,h_2$ are continuous. 
If $J$ is not open then there are 
$s\in J$ and $s_n\to s$ with $s_n\notin J$ for 
all $n\geq 1$.
Then $\phi_{h_i(s_n)}\in H(\phi_{s_n}(x))$ for $i=1,2$. 
This contradicts that $H(\phi_{s_n}(x))$ is a cross section. 
Then $J$ is closed and open, since $I$ is connected, $I=J$ and 
$h_1=h_2$.

We have that $h$ is increasing because $H$ is monotonous.
\end{proof}

In case such a function $h\colon I\to\R$ exists for $x,y$ define 
the \emph{sectional flow}
$$\Phi_t(x,y)=(\phi_t(x),\phi_{h(t)}(y))$$ for $t\in I$.
Define $\reparam^+$ (resp. $\reparam^-$) as the set of 
all homeomorphisms of $[0,+\infty)$ (resp. $(-\infty,0]$).
Define $W^s_\epsilon,W^u_\epsilon\colon X\to\hyper X$ as
\[
\begin{array}{l}
W^s_\epsilon(x)=\{y\in X :\exists h\in\reparam^+
\hbox{ such that } \phi_{h(t)}(y)\in H_\epsilon(\phi_t(x))
\forall t\geq 0\},\\
W^u_\epsilon(x)=\{y\in X :\exists h\in\reparam^-
\hbox{ such that } \phi_{h(t)}(y)\in H_\epsilon(\phi_t(x)) 
\forall t\geq 0\}.
\end{array}
\]

\begin{lem}
\label{hConvUnif}
If $h_n\colon I\to \R$ are continuous, $h_n(0)=0$, $x_n\to x$, $y_n\to y$ satisfy 
$\phi_{h_n(t)}(y_n)\in H(\phi_t(x_n))$ for all $n\geq 1$ and for all 
$t\in I$ then $h_n$ uniformly converges to some $h\colon I\to\R$ satisfying
$\phi_{h(t)}(y)\in H(\phi_t(x))$ for all $t\in I$.
\end{lem}

\begin{proof}
Suppose that $H$ has a time $\tau>0$ and 
that $h_n$ does not uniformly converge. 
Then there are $\epsilon>0$, $n_k,m_k\to+\infty$, $t_n\in I$ such that 
$|h_{n_k}(t_k)-h_{m_k}(t_k)|\geq \epsilon$ for all $k\geq 1$. 
Since $h_n(0)=h_m(0)=0$ for all $m,n\geq 1$ we can assume that 
$h_{n_k}(t_k)-h_{m_k}(t_k)=s\in (-\tau,\tau)$, $s\neq 0$, for all $k\geq 1$. 
Then $\phi_{h_{n_k}(t_k)}(y_{n_k})\in H(\phi_{t_k}(x_{n_k}))$ and
$\phi_{h_{m_k}(t_k)}(y_{m_k})\in H(\phi_{t_k}(x_{m_k}))$ for all $k\geq 1$. 
Assume that $\phi_{t_k}(x_{n_k}),\phi_{t_k}(x_{m_k})\to p$ 
and $\phi_{h_{n_k}(t_k)}(y_{n_k})\to q$. 
Then $\phi_{h_{m_k}(t_k)}(y_{m_k})\to \phi_s(q)$. 
We have a contradiction because $q$ and $\phi_s(q)$ are in $H(p)$ and $0<|s|<\tau$.
\end{proof}

\begin{prop}
  $W^s_\epsilon$ and $W^u_\epsilon$ are symmetric fields of compact sets.
\end{prop}

\begin{proof}
(Semicontinuity). It follows by Lemmas \ref{lemUniRep} and \ref{hConvUnif}.

(Symmetry). If $y\in W^s_\epsilon(x)$ then 
there is $h\in\reparam^+$ such that 
$\phi_{h(t)}(y)\in H_\epsilon(\phi_t(x))$ for all $t\geq 0$. 
We have that $h^{-1}\in\reparam^+$ and 
$\phi_{h^{-1}(t)}(x)\in H_\epsilon(\phi_t(y))$ for all $t\geq 0$. 
We have used that $H_\epsilon$ is locally symmetric. This proves that $W^s_\epsilon$ is symmetric. 
The proof for $W^u_\epsilon$ is similar. 
\end{proof}

\subsection{Expansive flows}
\label{secCwexpFaS}

Let $\phi\colon\R\times X\to X$ be a regular flow 
on a compact metric space $X$.
Recall that $\phi$ is an \emph{expansive flow} if 
for all 
$\epsilon>0$ there is $\delta>0$ such that if 
$\dist(\phi_{h(t)}(y),\phi_t(x))<\delta$ for all $t\in\R$ with 
$h\colon\R\to\R$ an increasing homeomorphism such that 
$h(0)=0$ then there is $t\in(-\epsilon,\epsilon)$ such that $y=\phi_t(x)$.
In this case $\delta$ is called \emph{expansive constant}.
\begin{prop} 
\label{propExpEquiv}
The following statements are equivalent:
\begin{enumerate}
\item $\phi$ is expansive,
\item there is $\delta>0$ such that 
$W^s_\delta(x)\cap W^u_\delta(x)=\{x\}$ for all $x\in X$.
\end{enumerate}
\end{prop}

\begin{proof}
(1 $\to$ 2). Consider that $H_\radio(x)$ is a local cross section of time $\tau$ for all $x\in X$. 
For $\epsilon\in(0,\tau)$ consider an expansive constant $\delta$ from the definition. 
Suppose 
that $y\in W^s_\delta(x)\cap W^u_\delta(x)$. 
So, there is $h\colon\R\to\R$ such that 
$\phi_{h(t)}(y)\in H_\delta(\phi_t(x))$ for all $t\in\R$.
Then $\dist(\phi_{h(t)}(y),\phi_t(x))<\delta$ for all $t\in\R$. 
The expansiveness of the flow implies that $y=\phi_t(x)$ for some $t\in(-\epsilon,\epsilon)$. 
Therefore we have that $y=x$ because $y\in H_\delta(x)$.

(2 $\to$ 1). Consider $\epsilon\in (0,\tau)$ given. 
Take $\delta>0$ such that for all $x\in X$ we have 
$B_\delta(x)\subset \phi_{[-\epsilon,\epsilon]}(H_\alpha(x))$. 
Suppose that $\dist(\phi_{g(t)}(y),\phi_t(x))<\delta$ for all $t\in\R$ and some $g\colon\R\to\R$. 
For $x\in X$ consider the flow projection 
$\pi_x\colon \phi_{[-\epsilon,\epsilon]}(H_\alpha(x))\to H_\alpha(x)$.
We have that 
$$\pi_{\phi_t(x)}(\phi_{g(t)}(y))\in H_\alpha(\phi_t(x))$$
for all $t\in\R$. 
If $y'=\pi_x(y)$ then we conclude $y'=x$, so, there is $t\in(-\epsilon,\epsilon)$ such that 
$y=\phi_t(x)$.
\end{proof}

Let us recall that a flow is \emph{positive expansive} 
if for all $\epsilon>0$ there is $\delta>0$ such that if 
$\dist(\phi_{h(t)}(y),\phi_t(x))<\delta$ for all $t\geq 0$ with 
$h\in\reparam^+$ then there is 
$t\in(-\epsilon,\epsilon)$ such that $y=\phi_t(x)$.

\begin{prop} 
\label{propPosExp}
For a regular flow $\phi$ on a compact metric space the following 
statements are equivalent:
\begin{enumerate}
\item \label{posexpit1} $\phi$ is positive expansive,
\item \label{posexpit2} there is $\delta>0$ such that 
$W^s_\delta(x)=\{x\}$ for all $x\in X$,
\item \label{posexpit3} $X$ is a finite union of circles.
\end{enumerate}
\end{prop}

\begin{proof}
The proof of the equivalence of \ref{posexpit1} and \ref{posexpit2} is analogous to the proof of Proposition 
\ref{propExpEquiv}. 
The equivalence of \ref{posexpit1} and \ref{posexpit3} is shown in \cite{Ar}.
\end{proof}

\subsection{Stable points}
\label{secStables}

The following definitions should be understood as \emph{Lyapunov stability 
of trajectories allowing time lags}.

\begin{df}
 A point $x\in X$ is \emph{stable} if for all $\epsilon>0$ there is $\delta>0$ 
 such that if $H_\delta(x)\subset W^s_\epsilon(x)$. 
 We say that $x\in X$ is \emph{asymptotically stable} 
 if it is stable and there is $\delta>0$ such that for all $y\in H_\delta(x)$ 
 there is $h\in\reparam^+$ such that $\phi_{h(t)}(y)\in H_\epsilon(\phi_t(x))$ 
 for all $t\geq 0$ and $\dist(\phi_{h(t)}(y),\phi_t(x))\to 0$ as $t\to+\infty$.
\end{df}

In the sequel we will say that $\delta>0$ is an \emph{expansive constant} if 
$W^s_\delta(x)\cap W^u_\delta(x)=\{x\}$ for all $x\in X$.

\begin{prop}
\label{estImpAsint}
  If $\phi$ is expansive with expansive constant $\delta$ and $x\in X$ then 
  $$\lim_{t\to +\infty}\diam(\Phi_t(x,W^s_\delta(x)))\to 0.$$
\end{prop}

\begin{proof}
By contradiction assume that 
there are $\gamma\in (0,\delta]$ and $t_n\to+\infty$ such that 
$\diam(\Phi_{t_n}(x,W^s_\delta(x)))\geq \gamma$. 
Denote by $x_n=\phi_{t_n}(x)$ and suppose that $x_n\to y$.
We can also assume that $\Phi_{t_n}(x,W^s_\delta(x))$ converges (in the Hausdorff metric) to 
a compact set $Y$ contained in $H_\delta(y)$.
In this way we have that $\diam(Y)\geq\gamma>0$ and 
$Y\subset W^s_\delta(y)\cap W^u_\delta(y)$. 
This contradicts the expansivity of $\phi$ and finishes the proof.
\end{proof}

A similar result holds for $W^u_\delta(x)$.

\begin{lem}
  \label{lemKato}
  If $\phi$ is expansive then for all $\epsilon>0$ there is $\delta>0$ such that if 
  $C\subset H(x)$ is a continuum such that 
  $x\in C$, $\diam(C)\leq\delta$ and $\diam(\Phi_s(x,C))\leq\delta$
  for some $s>0$ 
  then $\diam(\Phi_t(x,C))\leq\epsilon$ for all $t\in[0,s]$.
\end{lem}

\begin{proof}
  By contradiction assume that there are $\epsilon>0$, 
  sequences $x_n\in C_n$, $0\leq t_n\leq s_n$ 
  with $C_n$ a continuum in $H(x_n)$, 
  $\diam(C_n)\to 0$, $\diam(\Phi_{s_n}(x_n,C_n))\to 0$ 
  and $\diam(\Phi_{t_n}(x_n,C_n))\geq\epsilon$. 
  As in the previous proof, a limit 
  continuum of $\Phi_{t_n}(x_n,C_n)$ contradicts expansivity of $\phi$.
\end{proof}

Define the inverse flow $\phi^{-1}_t(x)=\phi_{-t}(x)$ and the sets:

\[
\begin{array}{l}
 \alim(x)=\{y\in X: \exists t_n\to-\infty\hbox{ such that }\phi_{t_n}(x)\to y\}\\
 \wlim(x)=\{y\in X: \exists t_n\to+\infty\hbox{ such that }\phi_{t_n}(x)\to y\}
\end{array}
\]
They are usually called $\alpha$-\emph{limit} and $\omega$-\emph{limit sets} respectively.

\begin{prop}
\label{propNoEst}
 If $X$ is a Peano continuum and 
 $\phi$ is an expansive flow on $X$ with a stable point 
 for $\phi$ or $\phi^{-1}$
 then $X$ is a circle.
\end{prop}

\begin{proof}
Assume that $x\in X$ is a stable point of $\phi^{-1}$.
Take $\epsilon>0$. 
Since $x$ is a stable point we know that there is $\delta>0$ 
such that $H_\delta(x)\subset W^u_\epsilon(x)$.
By Lemma \ref{lemKato} we have that there is $\sigma>0$ such that if
$0\leq s\leq t$ then
\[
 \Phi_{-s}(\phi_t(x),H_\sigma(\phi_t(x)))\subset H_\delta(\phi_{t-s}(x)).
\]
This implies that every point in 
$\wlim(x)$ is stable for $\phi^{-1}$. 
Then, $\phi$ restricted to $\wlim(x)$ is a positive expansive flow. 
Since $\wlim(x)$ is a connected set, by Theorem \ref{propPosExp}, 
we have that $\wlim(x)$ is a circle.
Since $x$ is a stable point we have that $x\in \wlim(x)$.
We have proved that every stable point is periodic. 
Since the set of stable points is open we conclude that $X$ is a periodic orbit 
and consequently, $X$ is a circle.
\end{proof}

In this context, a Peano continuum is \emph{trivial} if it is a circle.
Recall that $X$ cannot be a singleton because we are assuming that 
it admits a flow without singular points.

\begin{prop}
\label{propNonTrivalWs}
If $\phi$ is expansive with expansive constant $\epsilon$ and $X$ is a non-trivial Peano continuum then 
there is $\delta>0$ such that for all $x\in X$ there is a continuum $C\subset W^s_\epsilon(x)$ such that $x\in C$ 
and $\diam(C)\geq \delta$.  
\end{prop}

\begin{proof}
For the expansive constant $\epsilon$ consider 
$\delta$ from Lemma \ref{lemKato}. 
For $x\in X$ consider $z\in\wlim(x)$ and
take $t_k\to\infty$ such that 
$x_k=\phi_{t_k}(x)\to z$. 
We can assume that
$z\in H_\delta(x_k)$ for all 
$k\geq 1$. 
By Proposition \ref{propNoEst} we know that $z$ is not 
stable for $\phi^{-1}$. 
Then, there is $T>0$ such that 
$\diam(\Phi_{-T}(x_k,H_\delta(x_k)))>\epsilon.$
Suppose that $t_k>T$ for all $k\geq 1$. 
Consider $C_k\subset H_\delta(x_k)$ such that $C_k$ is a continuum, 
$x_k\in C_k$, $\diam(\Phi_{-t}(x_k,C_k))\leq\epsilon$ for all 
$t\in[0,t_k]$
and $\diam(\Phi_{-s_k}(x_k,C_k))=\epsilon$ for some $s_k\in[0,t_k]$.
By Lemma \ref{lemKato} we know that 
$\diam(\Phi_{-t}(x_k,C_k))\geq\delta$ for 
all $t\geq s_k$. 
Notice that $x\in \Phi_{-t_k}(x_k,C_k)$ for all $k\geq 1$. 
Then a limit continuum $C$ of the sequence 
$\Phi_{-t_k}(x_k,C_k)$ satisfies the thesis of the proposition.
\end{proof}

I learned the argument of the following
proof from J. Lewowicz in the setting of expansive homeomorphisms.

\begin{thm}
\label{teoKawa}
If $X$ is a Peano continuum admitting an expansive flow then 
no open subset of $X$ is homeomorphic with $\R^2$.
\end{thm}

\begin{proof}
Let $X$ be a Peano continuum with a non-singular flow $\phi$. 
Suppose that $x\in X$ has a neighborhood $U$ that is homeomorphic with $\R^2$.
In this case we have that a connected cross section $H_\epsilon(x)$ is a 
compact arc. 
By Proposition \ref{propNonTrivalWs} we have that 
$W^s_\epsilon(x)$ contains a non-trivial continuum. 
But, since the cross section is an arc, 
we have that it contains an interior point (with respect to the topology of the arc). 
Then there are stable points.  
This contradicts Proposition \ref{propNoEst}.
\end{proof}

\begin{cor}[\cite{HS}]
No compact surface admits an expansive flow without singular points.
\end{cor}

\begin{proof}
  It is a direct consequence of Theorem \ref{teoKawa}.
\end{proof}

\begin{cor}[\cite{Kaw88}]
If a Peano continuum admits an expansive homeomorphism then 
no open set is homeomorphic with $\R$. 
In particular the circle and the interval do not admit expansive homeomorphisms.
\end{cor}

\begin{proof}
  It follows by Theorem \ref{teoKawa} recalling that a 
  homeomorphism is expansive if and only if its suspension is an expansive flow (see \cite{BW}).
\end{proof}


\begin{bibdiv}
\begin{biblist}

\bib{Ar}{article}{
author={A. Artigue},
title={Positive expansive flows},
journal={Topology Appl.},
volume={165},
year={2014},
pages={121--132}}



\bib{BP}{book}{
author={C. Bessaga},
author={A. Pelczynski},
title={Selected topics in infinite-dimensional topology},
year={1975},
publisher={Polish Scientific Publishers}}

\bib{Bing}{article}{
author={R. H. Bing},
title={Partitioning a set},
journal={Bull. Amer. Math. Soc.},
volume={55}, 
year={1949}, 
pages={1101--1110}}

\bib{BW}{article}{
author={R. Bowen and P. Walters}, title={Expansive One-Parameter
Flows}, journal={J. Differential Equations}, year={1972}, pages={180--193},
volume={12}}

\bib{Goodman}{article}{
journal={Topology},
volume={24},
year={1985}, 
pages={333--340},
title={Vector fields with transverse foliations},
author={S. Goodman}}


\bib{HS}{article}{
author={L. F. He},
author={G. Z. Shan},
title={The nonexistence of expansive flow on a compact 2-manifold},
journal={Chin. Ann. Math. Ser. B}, 
volume={12},
year={1991}, 
pages={213--218}}


\bib{IN}{book}{
author={A. Illanes},
author={S. B. Nadler Jr.},
title={Hyperspaces: Fundamentals and Recent Advances},
year={1999},
publisher={Marcel Dekker, Inc.}}

\bib{Ka93}{article}{
author={H. Kato},
title={Continuum-wise expansive homeomorphisms},
journal={Canad. J. Math.},
volume={45},
number={3},
year={1993},
pages={576--598}}

\bib{Kaw88}{article}{
author={K. Kawamura},
title={A direct proof that each Peano continuum with a free arc admits no expansive homeomorphism},
journal={Tsukuba J. Math.},
volume={12},
pages={521--524},
year={1988}}

\bib{KS}{article}{
author = {H. B. Keynes},
author={M. Sears},
title = {Real-expansive flows and topological dimension},
journal = {Ergodic Theory Dynam. Systems},
volume = {1},
year = {1981},
pages = {179--195},}

\bib{Lew}{book}{
author={J. Lewowicz},
title={Lyapunov functions and stability of geodesic flows},
year={1983},
publisher={Springer},
series={Lecture Notes in Math.},
volume={1007},
pages={463--479}}




\bib{Milnor}{article}{
author={J. Milnor}, 
title={Microbundles Part I}, 
journal={Topology},
volume={3},
year={1964}, 
pages={53--80}}

\bib{Moise}{article}{
author={E. E. Moise}, 
title={Grille decomposition and convexification theorems for compact 
metric locally connected continua},
journal={Bull. Amer. Math. Soc.},
volume={55},
year={1949},
pages={1111-1121}}

\bib{MSS}{article}{
author={K. Moriyasu},
author={K. Sakai},
author={W. Sun},
title={$C^1$-stably expansive flows},
journal={J. Differential Equations},
volume={213},
year={2005},
pages={352--367}}

\bib{Oka}{article}{
author={M. Oka},
title={Singular foliations on cross-sections of expansive flows on 3-manifolds},
journal={Osaka J. Math.},
volume={27},
pages={863--883},
year={1990}}

\bib{Pat}{article}{
author={M. Paternain},
title={Expansive Flows and the Fundamental Group},
journal={Bull. Braz. Math. Soc.},
year={1993},
volume={24},
number={2},
pages={179--199}
}

%

\bib{Th87}{article}{
author={R. F. Thomas},
title={Entropy of expansive flows}, 
journal={Ergodic Theory Dynam. Systems},
year={1987},
pages={611--625},
volume={7}}

\bib{W}{article}{
author={H. Whitney},
title={Regular Family of Curves},
journal={Ann. of Math.},
year={1933},
volume={34},
number={2},
pages={244--270}}

\end{biblist}
\end{bibdiv}
\end{document}